\documentclass[a4paper, final, 11pt]{amsart}

\usepackage[latin1]{inputenc}
\usepackage[T1]{fontenc}
\usepackage[english,french]{babel}
\usepackage{amsmath,amssymb,amscd,latexsym}
\usepackage[active]{srcltx}
\usepackage{graphicx}
\usepackage{url}

\numberwithin{equation}{section}

\usepackage[draft]{showkeys}

\theoremstyle{plain}
\newtheorem{Thm}{Th\'eor\`eme}[section]
\newtheorem{Prop}[Thm]{Proposition}

\newtheorem{theorem}[Thm]{Theorem}
\newtheorem{corollary}[Thm]{Corollary}
\newtheorem{lemma}[Thm]{Lemma}
\newtheorem{proposition}[Thm]{Proposition}

\newtheorem{Claim}[Thm]{Claim}

\theoremstyle{definition}

\newtheorem{Def}[Thm]{D\'efinition}

\newtheorem{definition}[Thm]{Definition}

\newtheorem{Remark}[Thm]{Remark}

\theoremstyle{remark}

\newcommand{\diam}{\operatorname{diam}}

\newcommand{\id}{\operatorname{id}}

\newcommand{\stab}{\operatorname{Stab}}

\newcommand{\sym}{\operatorname{Sym}}
\newcommand{\shadow}{\operatorname{Shadow}}

\newcommand{\Ac}{\mathcal A}

\newcommand{\Fc}{\mathcal F}

\newcommand{\Hc}{\mathcal H}

\newcommand{\Uc}{\mathcal U}

\newcommand{\Kbar}{\overline{K}}

\newcommand{\Ybar}{\overline{Y}}

\newcommand{\N}{\mathbf{N}}
\newcommand{\Z}{\mathbf{Z}}
\newcommand{\R}{\mathbf{R}}
\newcommand{\F}{\mathbf{F}}
\newcommand{\Out}{\mathrm{Out}}
\newcommand{\Aut}{\mathrm{Aut}}
\newcommand{\Inn}{\mathrm{Inn}}
\newcommand{\SL}{\mathrm{SL}}
\newcommand{\PSL}{\mathrm{PSL}}
\newcommand{\GL}{\mathrm{GL}}

\begin{document}
\selectlanguage{english}

\title[Highly transitive actions of free products]{Highly transitive actions of free products}
\author{Soyoung Moon}
\author{Yves Stalder}
\thanks{The second-named author has been partially supported by ANR grant ``AGORA''}
\date{\today}
\address{Soyoung Moon \newline
Université de Bourgogne, Institut de Mathématiques de Bourgogne, UMR 5584 du CNRS, B.P.
47870, 21078 Dijon cedex, France}
\email{soyoung.moon@u-bourgogne.fr}
\address{Yves Stalder \newline Clermont Universit\'e, Universit\'e Blaise Pascal, Laboratoire de
Math\'ematiques, BP 10448,\newline F-63000 Clermont-Ferrand, France\newline
CNRS, UMR 6620, Laboratoire de Math\'ematiques, F-63177 Aubi\`ere, France}
\email{yves.stalder@math.univ-bpclermont.fr}

\begin{abstract}
 We characterize free products admitting a faithful and highly transitive action. In particular, we show that the group $\PSL_2(\Z)\simeq (\Z/2\Z)*(\Z/3\Z)$ admits a faithful and highly transitive action on a countable set.
\end{abstract}

\maketitle

\section*{Introduction}
Let $X$ be a countable\footnote{In this paper, ``countable'' means ``infinite countable''.} set and let $G$ be a countable group acting on $X$. The action is called \emph{highly transitive} if, for all $k\in\N^*$, it is transitive on
ordered $k$-tuples of distinct elements.

Dixon proved \cite{Dixon} that for any integer $k\geq 2$, generically in Baire's sense, $k$ permutations $x_1,\ldots,x_k\in\sym(\N)$ such that the subgroup $\langle x_1,\ldots,x_k \rangle$ acts without finite orbits generate a free group of rank $k$ which acts highly transitively on $\N$. Adapting this approach, Kitroser \cite{Kitroser} showed that the fundamental groups of surfaces of genus at least 2 admit a faithful and highly transitive action.

Garion and Glasner \cite{GarionGlasner} proved that for $n \geq 4$ the group of outer automorphisms of free group on $n$ generators $\Out(\F_n) = \Aut(\F_n)/ \Inn(\F_n)$
admits a faithful and highly transitive action. They asked whether $\Out(\F_2)\simeq \GL_2(\Z)$ and $\Out(\F_3)$ admit a highly transitive action.
In this paper, with methods in Dixon's spirit, we prove the following result.
\begin{theorem}\label{main}
 Let $G,H$ be non-trivial finite or countable groups. Then, the following statements are equivalent:
 \begin{enumerate}
  \item the free product $G*H$ admits a faithful and highly transitive action;
  \item at least one of the factors $G,H$ is not isomorphic to the cyclic group $\Z/2\Z$.
 \end{enumerate}
\end{theorem}
In particular, the group $\PSL_2(\Z)\simeq (\Z/2\Z)*(\Z/3\Z)$ admits a faithful and highly transitive action. As a consequence, the group $\SL_2(\Z)$ admits a highly transitive action on a countable set. On the other hand, this group cannot admit faithful and highly transitive actions since it has non-trivial center (see Corollary \ref{GeneralFactsFHT}).

The paper is organized as follows. Section 1 contains preliminaries about highly transitive actions and Baire's theory. Sections 2 and 3 are devoted to the proof of Theorem \ref{main}.

We thank Georges Skandalis for suggesting Corollary \ref{NotSolvable} and Alain Valette for pointing out a small mistake in the last proof in the first version of this paper.

\section{Preliminaries}\label{Preliminaries}
\subsection{Generalities about group actions}
Let us begin with a general fact concerning actions with infinite orbits.
\begin{lemma}\textbf{\emph{(B. H. Neumann, P. Neumann)}}\label{DeplacementPartiesFinies}
Let $G$ be a group acting on some set $X$ and $F$ be a finite subset of $X$. If every orbit of the points in $F$ is infinite, then there exists $g\in G$ such that $g\cdot F \cap F = \emptyset$.
\end{lemma}

\begin{proof}
 This follows from Lemma 4.1 in \cite{NeumannCovered} and Lemma 2.3 in \cite{NeumannFinitary}. Indeed, let us suppose that for every $g\in G$, $gF\cap F\neq \emptyset$. If we denote $K_{xy}:=\{g\in G | gx=y\}$, $\forall x$, $y\in F$, then by hypothesis we have $G=\bigcup_{x,y\in F}K_{xy}$. When $K_{xy}\neq \emptyset$, we have $K_{xy}=\stab(y)g_{xy}$ with some $g_{xy}\in K_{xy}$. Then
 $$
 G=\bigcup_{\text{$x,y\in F$ such that $K_{xy}\neq \emptyset$}}\stab(y)g_{xy}.
 $$
 Then by Lemma 4.1 in \cite{NeumannCovered}, there exists $y\in F$ such that the index of $\stab(y)$ is finite. Therefore the orbit $Gy$ is finite.
\end{proof}

From the above lemma, it follows immediately that:
\begin{Remark}\label{DeplacementDeuxPartiesFinies}
Let $X$ be a $G$-set and $F_1$, $F_2$ be finite subsets of $X$. If every orbit of the points in $F_1$ and $F_2$ are infinite, then there exists $g\in G$ such that $g\cdot F_1 \cap F_2 = \emptyset$.
\end{Remark}

\subsection{Highly transitive actions}
Let $G$ be a group acting on some set $X$. Let us recall that the action is called \emph{faithful} if the corresponding homomorphism $G\to \sym(X)$ is injective and \emph{transitive} if for any $x,y \in X$, there exists $g\in G$ such that $g\cdot x = y$. Given a positive integer $k$, we set
\[
 X^{(k)} = \big\{(x_1,\ldots, x_k)\in X^k : \  x_i \neq x_j \text{ for all } i\neq j\big\} \ ,
\]
and the action $G\curvearrowright X$ is called \emph{$k$-transitive} if the diagonal $G$-action on $X^{(k)}$ is transitive.
\begin{definition}\label{HighlyTransitive}
 Assume that $G$ and $X$ are countable.
 The action $G\curvearrowright X$ is called highly transitive if it is $k$-transitive for any positive integer $k$.
\end{definition}
Defining highly transitive actions on a finite set $Y$ would not be interesting, since $Y^{(k)}$ is empty for all $k>|Y|$.

We are interested to determine which groups admit highly transitive actions respectively faithful and highly transitive actions. Here are some general facts, which are probably well-known by experts; see e.g. \cite[Section 5.1]{GarionGlasner} for item (2).
\begin{proposition}\label{GeneralFactsIHT}
 Let $G\curvearrowright X$ be a highly transitive action. Then:
 \begin{enumerate}
  \item any central element of $G$ acts trivially;
  \item for any normal subgroup $K\lhd G$, the action $K\curvearrowright X$ is either trivial, or highly transitive.
  \item for any finite index subgroup $H<G$, the action $H\curvearrowright X$ is highly transitive.
 \end{enumerate}
\end{proposition}

\begin{proof}
(1) Let $g$ be an element of $G$ which acts non-trivially and let $x_1\in X$ such that $x_1$ and $x_2:=gx_1$ are distinct. Let $y_1, y_2\in X$ such that $y_2$ is distinct from $y_1$ and $gy_1$ (this is possible since $X$ is infinite). Then, by high transitivity, there is an element $h \in G$ such that $hx_1=y_1$ and $hx_2=y_2$. We have
\[
 hgx_1 = hx_2 = y_2 \quad \text{and} \quad ghx_1 = gy_1\neq y_2 \ ,
\]
which proves that $g$ is not a central element.

(2) Suppose that the action is not trivial, i.e. that there exists $x\in X$ and $k\in K$ such that $x\neq kx$. For any $y\in X$ different from $x$, there exists $g\in G$ such that $gx=x$ and $gy=kx$. Then $g^{-1}kgx=y$ and therefore $y$ is in $K\cdot x$ by normality of $K$ in $G$. This proves that the action $K\curvearrowright X$ is transitive.

Let $x=(x_1,\ldots,x_k)$ and $y=(y_1,\ldots,y_k)$ in $X^{(k)}$. By Lemma \ref{DeplacementPartiesFinies}, there exists $h\in K$ such that
$$\{hy_1,\ldots,hy_k\} \cap \big(\{y_1,\ldots,y_k\}\cup \{x_1,\ldots,x_k\}\big)=\emptyset.
 $$
 Then $(x_1,\ldots, x_k, hy_1,\ldots, hy_k)$ is in $X^{(2k)}$. Take $(z_1,\ldots,z_k)\in X^{(k)}$. Again by Lemma \ref{DeplacementPartiesFinies}, there exists $h'\in K$ such that $\{h'z_1,\ldots,h'z_k\}\cap \{z_1,\ldots,z_k\}=\emptyset$. Then $(z_1,\ldots,z_k, h'z_1,\ldots,h'z_k)$ is in $X^{(2k)}$. Since the $G$-action on $X$ is highly transitive, there exists $g\in G$ such that
$$
g(x_1,\ldots, x_k, hy_1,\ldots, hy_k)=(z_1,\ldots,z_k, h'z_1,\ldots,h'z_k).
$$
Then $z_i=gx_i$ and $ghy_i=h'z_i=h'gx_i$, so
$$
y_i=h^{-1}g^{-1}h'gx_i,
$$
for every $i=1,\ldots,k$. Since $K$ is normal in $G$, the element $h^{-1}g^{-1}h'g$ is in $K$ and therefore $K\curvearrowright X$ is highly transitive.

(3) There exists a normal subgroup $K\lhd G$, contained in $H$, which has finite index in $G$. It cannot act trivially since $[G:K]$ is finite and the unique $G$-orbit is infinite. Thus the assertion follows from (2).
\end{proof}

For faithful and highly transitive actions, we have the following straightforward corollary:

\begin{corollary}\label{GeneralFactsFHT}
 Assume that $G\curvearrowright X$ is a faithful and highly transitive action. Then:
 \begin{enumerate}
  \item the center $Z(G)$ is trivial;
    \item for any non-trivial normal subgroup $K\lhd G$, the action $K\curvearrowright X$ is faithful and highly transitive;
  \item for any finite index subgroup $H<G$, the $H$-action on $X$ is faithful and highly transitive.
 \end{enumerate}
\end{corollary}


\begin{corollary}\label{NotSolvable}
 If $G\curvearrowright X$ is a faithful and highly transitive action, then $G$ is not solvable.
\end{corollary}
\begin{proof}
 For any $n\in \N$, the $n$-th derived subgroup $G^{(n)}$ is a normal subgroup of $G$. If $G^{(k)}$ is non-trivial, then it acts highly transitively on $X$ by Corollary \ref{GeneralFactsFHT} (2), so that it is non-abelian, by Corollary \ref{GeneralFactsFHT} (1). Hence  $G^{(k+1)}$ is non-trivial. This proves (by induction) that $G$ is not solvable.
\end{proof}

Notice that if $G$ contains a finite index subgroup which admits a faithful and highly transitive action, this does \textbf{not} imply that $G$ itself admits a faithful and highly transitive action. For example, $\SL_2(\Z)$ has a free subgroup of index $12$, but does not admit a faithful and highly transitive action since its center is non-trivial.

\subsection{Baire spaces}\label{PrelimBaire}
Let $X$ be a countable set. For any enumeration $X=\{x_0,x_1,x_2,\ldots\}$, one can consider the distance on the group $\sym(X)$ defined by
\[
 d(\sigma,\tau) = 2^{- \inf\{k\in\N :\, \sigma(x_k) \neq \tau(x_k) \text{ or } \sigma^{-1}(x_k) \neq \tau^{-1}(x_k)\}} \, .
\]
Then, $\sym(X)$ becomes a complete ultrametric space and a topological group. Note that a sequence $(\sigma_n)$ in $\sym(X)$ converges to a permutation $\sigma$ if and only if, given any finite subset $F\subset X$, the permutations $\sigma$ and $\sigma_n$, respectively $\sigma^{-1}$ and $\sigma_n^{-1}$, coincide on $F$ for $n$ large enough. Hence the topology on $\sym(X)$ is independent of the chosen enumeration. One can notice that a subgroup $\Gamma$ of $\sym(X)$ is dense if and only if the $\Gamma$-action on $X$ is highly transitive.

As a complete metrizable space, $\sym(X)$ is a \emph{Baire space}, that is a topological space in which every countable intersection of dense open subsets is still dense. In such a space, a countable intersection of dense open subsets is called \emph{generic subset}, or \emph{co-meager subset}, while its complement (that is countable union of closed sets with empty interior) is called \emph{meager subset}. In particular, generic subsets are dense, thus non-empty.

The case of free products $G*H$ with two infinite factors (see Section \ref{InfiniteFactors}) will be treated by genericity arguments in $\sym(X)$. For the case of free products $G*H$ with a finite factor, we need to consider a clever Baire space that we introduce now.
Let us consider two non-trivial finite or countable groups $G,H$ and assume that $X$ is endowed with some $G$-action such that it is isomorphic (in the category of $G$-sets) to $G\times \N$, where $G$ acts by left multiplication on the first factor. The product $\sym(X)^H$ admits the following complete metric:
\[
d((\sigma_h)_{h\in H},(\tau_h)_{h\in H})= \max\{d(\sigma_h,\tau_h):\, h\in H\} \, ,
\]
where $\sym(X)$ is endowed with the metric defined above. One can again see that the topology on $\sym(X)^H$ does not depend on the choice of an enumeration of $X$. Moreover, when $H$ is finite, this topology coincides with the product topology. The set of $H$-actions on $X$ identifies then with the subset $\operatorname{Hom}(H,\sym(X))\subset \sym(X)^H$. It is easy to check that this subset is closed in $\sym(X)^H$, hence is a complete metrizable space.
\begin{definition}
 Let $X$ be a $G$-set. We call an action $\sigma: H\to\sym(X)$ \textit{admissible} if all orbits $\langle G,\sigma(H) \rangle$ in $X$ are infinite.

 The set of admissible actions will be denoted by $\Ac(G,H,X)$.
\end{definition}
Notice that $\Ac(G,H,X)$ is non-empty. Indeed, if we identify $\N$ to $G\backslash(G*H)$ (which is indeed countable), $X$ is identified (as a $G$-set) to $G*H$. Then the $H$-action by left multiplications on $G*H$ corresponds to a $H$-action on $X$ which is admissible.
\begin{lemma}
 The space $\Ac(G,H,X)$ is a complete metrizable space.
\end{lemma}
In particular, the space $\Ac(G,H,X)$  is a Baire space.
\begin{proof}
 It suffices to check that $\Ac(G,H,X)$ is closed in $\operatorname{Hom}(H,\sym(X))$. To do so, let us consider a sequence $(\sigma_n)_{n\in\N}$ in $\Ac(G,H,X)$ which converges to $\sigma \in \operatorname{Hom}(H,\sym(X))$ and prove that $\sigma$ is an admissible action.
 Firstly, if we assume that $F$ is a finite orbit of the subgroup $\langle G,\sigma(H) \rangle$, then for $n$ large enough, the components of $\sigma_n$ (and their inverses) would coincide with the components of $\sigma$ (and their inverses) on $F$ and $F$ would be a finite orbit of the subgroup $\langle G,\sigma_n(H) \rangle$, which is impossible since $\sigma_n$ is an admissible action.
 Secondly, if $x_1,x_2$ are two distinct points in the intersection of some $G$-orbit and some $\sigma(H)$-orbit, then for $n$ large enough, $x_1$ and $x_2$ would belong to the intersection of some $G$-orbit and some $\sigma_n(H)$-orbit, which is again impossible since $\sigma_n$ is an admissible action.
\end{proof}

\section{Case with two infinite factors}\label{InfiniteFactors}
The aim of this section is to prove the following result.
\begin{theorem}\label{FreeProductInfinite}
If $G$ and $H$ are countable groups, then the free product $G\ast H$ admits a faithful and highly transitive action.
\end{theorem}
It will be a direct consequence of two propositions in the following setting. Let $X$ be a countable set and let $G,H$ be two subgroups of $\sym(X)$. For any $\sigma\in \sym(X)$, let us consider the action ${\phi_{\sigma}:G\ast H\rightarrow \sym(X)}$ defined by
\[
\phi_{\sigma}(w)=w^{\sigma}:=g_1\sigma^{-1}h_1\sigma\cdots g_k\sigma^{-1}h_k\sigma
\]
where $w=g_1h_1\cdots g_k h_k$ with $g_1,\ldots,g_k\in G$ and $h_1,\ldots,h_k\in H$.

\begin{Prop}\label{FreeProductInfiniteOrbit}
Suppose that every orbit of $G$ and $H$ on $X$ is infinite. Then
\[
 \Hc := \{\sigma\in\sym(X): \phi_\sigma \text{ is highly transitive}\}
\]
is generic in $\sym(X)$.
\end{Prop}

\begin{Prop}\label{FreeProductFaithful}
Suppose that every non trivial element of $G$ and $H$ has infinite support. Then the set
$$
\Fc=\{\sigma\in \sym(X) \mid \phi_\sigma \textrm{ is faithful } \}
$$
is generic in $\sym(X)$.
\end{Prop}

\begin{proof}[Proof of Theorem \ref{FreeProductInfinite} based on the propositions.]
 Let $G,H$ be countable groups; let $X$ be the countable set considered above. One can endow $X$ with a $G$-action and a $H$-action which are both transitive and free. Then, $G$ and $H$ can be identified with their images in $\sym(X)$.
 Moreover, by Propositions \ref{FreeProductInfiniteOrbit} and \ref{FreeProductFaithful}, we can take a permutation $\sigma\in \Hc\cap\Fc$ (in fact, $\Hc\cap\Fc$ is generic in $\sym(X)$); the $G*H$-action $\phi_\sigma$ is then highly transitive and faithful.
\end{proof}

\begin{proof}[Proof of Proposition \ref{FreeProductInfiniteOrbit}.]
 For every $k\in\N^*$ and $x=(x_1,\ldots,x_k)$, $y=(y_1\ldots,y_k)\in X^{(k)}$, let
\[
 U_{k,x,y} = \{\sigma\in\sym(X): \, \exists w\in G*H \text{ such that } w^\sigma(x_i)=y_i
\ \forall i=1,\ldots,k\}  \ .
\]
Since $\Hc = \bigcap_{k\in\N^*}\bigcap_{x,y\in X^{(k)}}U_{k,x,y}$, it is enough to prove that the set $U_{k,x,y}$ is open and dense.

Let $\sigma\in  U_{k,x,y}$ and let $w$ such that $ w^{\sigma}(x_i)=y_i$ for every
$i=1,\ldots,k$. The map $\sigma\mapsto w^\sigma$ is continuous and the inverse image of the open set
$\{\alpha\in\sym(X): \, \alpha(x_i)=y_i \ \forall i=1, \ldots,k\}$ contains $\sigma$ and is contained in $U_{k,x,y}$. Thus the set $U_{k,x,y}$ is a neighborhood
of $\sigma$ and this shows that $U_{k,x,y}$ is open.

Let us show that $U_{k,x,y}$ is dense. Let $F\subset X$ be a finite subset of $X$ and $\tau\in \sym(X)$. Given a subset $Y\subseteq X$, we denote by $\tau^{\pm 1}(Y)$ the union $\tau(Y)\cup\tau^{-1}(Y)$. Let
$I=\{x_1,\ldots,x_k\}$ and $J=\{y_1,\ldots,y_k\}$. We start by a variation of Remark \ref{DeplacementDeuxPartiesFinies}.
\begin{Claim}
 For any finite subsets $A,B$ of $X$, there exists $g\in G$ such that
 \[
  \big(gA\cup \tau^{\pm 1}(gA)\big) \cap \big(B \cup \tau^{\pm 1}(B)\big) = \emptyset \, .
 \]
 Similarly, there exists $h\in H$ such that $\big(hA\cup \tau^{\pm 1}(hA)\big) \cap \big(B \cup \tau^{\pm 1}(B)\big) = \emptyset$.
\end{Claim}
Indeed, set $B' = B \cup \tau^{\pm 1}(B)$. By Remark \ref{DeplacementDeuxPartiesFinies}, there exists $g\in G$ (respectively $h\in H$) such that $gA \cap B' = \emptyset$ and $gA \cap  \tau^{\pm 1}(B') = \emptyset$. This implies $gA \cap B' = \emptyset$ and $\tau^{\pm 1}(gA) \cap B' = \emptyset$, hence $\big(gA\cup \tau^{\pm 1}(gA)\big) \cap \big(B \cup \tau^{\pm 1}(B)\big) = \emptyset$. The Claim is proved.

Hence, there exists $g_1\in G$ such that $(F\cup \tau^{\pm 1} F) \cap (g_1 I\cup \tau^{\pm 1}g_1 I)=\emptyset$. Then, taking $A=J$ and $B=F\cup g_1 I$, the Claim shows that there exists $g_2\in G$ such that the sets $F\cup\tau^{\pm 1}(F)$, $g_1I\cup \tau^{\pm 1} (g_1I)$, and $g_2J\cup\tau^{\pm 1}(g_2J)$ are pairwise disjoint.
We then choose a finite subset ${M=\{z_1,\ldots,z_k\}\subset X}$ such that the set $M\cup \tau^{\pm 1} M$ is disjoint from the finite sets considered so far.
Again by the Claim (with $A=M$ and $B= F\cup g_1I \cup g_2J \cup M$), there exists $h\in H$ such that the sets
$$
 F\cup\tau^{\pm 1}(F), \,\, g_1I\cup \tau^{\pm 1} (g_1I), \,\, g_2J\cup\tau^{\pm 1}(g_2J), \,\, M\cup \tau^{\pm 1} M \text{ and } h(M\cup \tau^{\pm 1} M)
$$
are pairwise disjoint.

We then define a permutation $\sigma$ of $X$ by
$$
 \sigma(g_1x_j)= z_j, \,\, \sigma(\tau^{-1}(z_j))=\tau(g_1x_j)
$$
$$
\sigma(g_2(y_j))= h(z_j), \,\, \sigma(\tau^{-1}(h(z_j)))=\tau(g_2(y_j))
$$
for every $j=1,\dots,k$, and $\sigma(x):=\tau(x)$ for every other points of $X$. In particular, $\sigma|_F=\tau|_F$ and $(g_2^{-1}hg_1)^\sigma(x_i)=y_i$ for all $i=1,\ldots,k$. This shows that $\sigma\in U_{k,x,y}$ and the set $U_{k,x,y}$ is dense.
\end{proof}

\begin{proof}[Proof of Proposition \ref{FreeProductFaithful}.]
This follows from the genericity of $\mathcal{O}_1$ in \cite{Moon}; here we give a self-contained proof in the case of free products.

For every $w\in G*H$, let $U_w=\{\sigma\in\sym(X): w^\sigma\neq \id_X\}$.
We have
\[
 \Fc=\bigcap_{w\in G*H\setminus\{1\}} U_w \ .
\]
So it is enough to show that for every $w\in G*H\setminus\{1\}$, the set
$U_w$ is open and dense.

It is clear that $U_w$ is open. Let us show that $U_w$ is dense. If $w$ is a nontrivial element of $G$ or $H$, then $U_w=\sym(X)$ since $G$ and $H$ act faithfully on $X$. If $w\notin G\cup H$ and $w\neq gh$ (with $g\in G\setminus \{1\}$ and $h\in H\setminus\{1\}$), then we can write
\[
 w = g_kh_k \cdots g_1h_1
\]
with $k\geq 2$, $g_k\in G$, $g_{k-1},\ldots,g_1\in G\setminus\{1\}$, $h_k,\ldots,
h_2\in H\setminus\{1\}$ and $h_1\in H$.

Let $\sigma'\in\sym(X)$ and let $F$ be a finite subset of $X$. Since the elements $g_1,\ldots,g_{k-1}$, $h_2,\ldots,h_k$ have infinite supports, there exist $x_0,\ldots,x_{2k-1},y_1,\ldots,y_{2k}\in X$ such that:
\begin{enumerate}
 \item[$\bullet$] none of these points are in $F\cup\sigma'^{\pm 1}(F)$;
 \item[$\bullet$] these points are pairwise disjoint, except possibly $x_0=x_1$ and $y_{2k}=y_{2k-1}$;
 \item[$\bullet$] for every $j=0,\ldots,k-1$, we have $h_{j+1}(x_{2j})= x_{2j+1}$;
 \item[$\bullet$] for every $j=1,\ldots,k$, we have $g_j(y_{2j-1})=y_{2j}$.
\end{enumerate}

If $x_0=x_1$, put $y_0=y_1$; if not, put $y_0=x_0$. Then put $\sigma(y_i)=x_i$
for every $i=0,\ldots,2k-1$ and $\sigma(x)=\sigma'(x)$ for all $x\in F$. This defines a bijection between $F\cup\{y_0,
\ldots y_{2k-1}\}$ and $\sigma(F)\cup\{x_0,\ldots,x_{2k-1}\}$. By extending the definition of $\sigma$ to the other points, we thus obtain a permutation $\sigma\in\sym(X)$ such that $\sigma|_F=\sigma'|_F$ and $w^\sigma(y_0)=y_{2k}\neq y_0$. In case where $w=gh$ with $g\in G\setminus\{1\}$ and $h\in H\setminus\{1\}$, there exist pairwise disjoint points $y_0,x_0,x_1,y_1,y_2$ outside of $F\cup \sigma'^{\pm 1}(F)$ such that $hx_0=x_1$ and $gy_1=y_2$. Then we define a permutation $\sigma \in \sym(X)$ such that $\sigma(y_0)=x_0$, $\sigma(y_1)=x_1$ and $\sigma|_F=\sigma'|_F$ so that $w^\sigma(y_0)=y_2\neq y_0$. This proves that $\sigma\in U_w$ and therefore $U_w$ is dense in $\sym(X)$.
\end{proof}

\section{Case with one finite factor}\label{FiniteFactor}
\subsection{Definitions and notations}
Let $G$, $H$ two non-trivial finite or countable groups. In this section, the set $X$ will be identified with the disjoint union of a countable collection of copies of $G$:
\[
 X = \bigsqcup_{j\in\N} G_j, \quad \text{ where } G_j=G \text{ for every } j .
\]

First of all, we give some definitions and fix the notations.
Given an action $\sigma:H\rightarrow \sym(X)$ and $G \curvearrowright X$, this induces an action of $G\ast H$ on $X$. Denote by $X_{\sigma}$ the Schreier graph of this action with respect to the generating set $G\cup H$ and by $d_{\sigma}$ the distance on $X_{\sigma}$. Given $u\in G*H$, we denote by $u^\sigma$ the image of $u$ in the subgroup $\langle G,\sigma(H) \rangle$ of $\sym(X)$.
\begin{Def}\label{defTrajectoire}
 Let $w\in G*H$ and $x\in X$. We call \emph{$\sigma$-trajectory of $w$ from $x$} the sequence $$(x,s_1(w)^\sigma(x),\ldots, s_{|w|-1}(w)^\sigma(x), w^\sigma(x)),$$
 where $s_j(w)$ is the suffix of $w$ of length $j$ (that is, if $w=w_{|w|}w_{|w|-1}\cdots w_2 w_1$ is written as a normal form, then $s_j(w)=w_j w_{j-1}\cdots w_2 w_1$).
\end{Def}

Consider the graph where the vertices are the right cosets $Gw$ and $Hw$, with $w\in G*H$, and the edges are the elements of $G*H$, such that the edge $w$ links two vertices $Gw$ and $Hw$. Recall that \cite{SerreArbres} this is a tree, called \emph{Bass-Serre tree} of $G*H$ and denote by $T$ its geometric realization (which is a real tree). Endowed with the right invariant word metric with respect to the generating set $G\cup H$, the map of $G*H$ in $T$ which sends an element $w$ on the middle point between the vertices $Gw$ and $Hw$ is an isometric injection. From now on, we will identify $G\ast H$ with the image (cf. Figure \ref{BStree}).

\begin{figure}
\centering
\includegraphics[width=12cm]{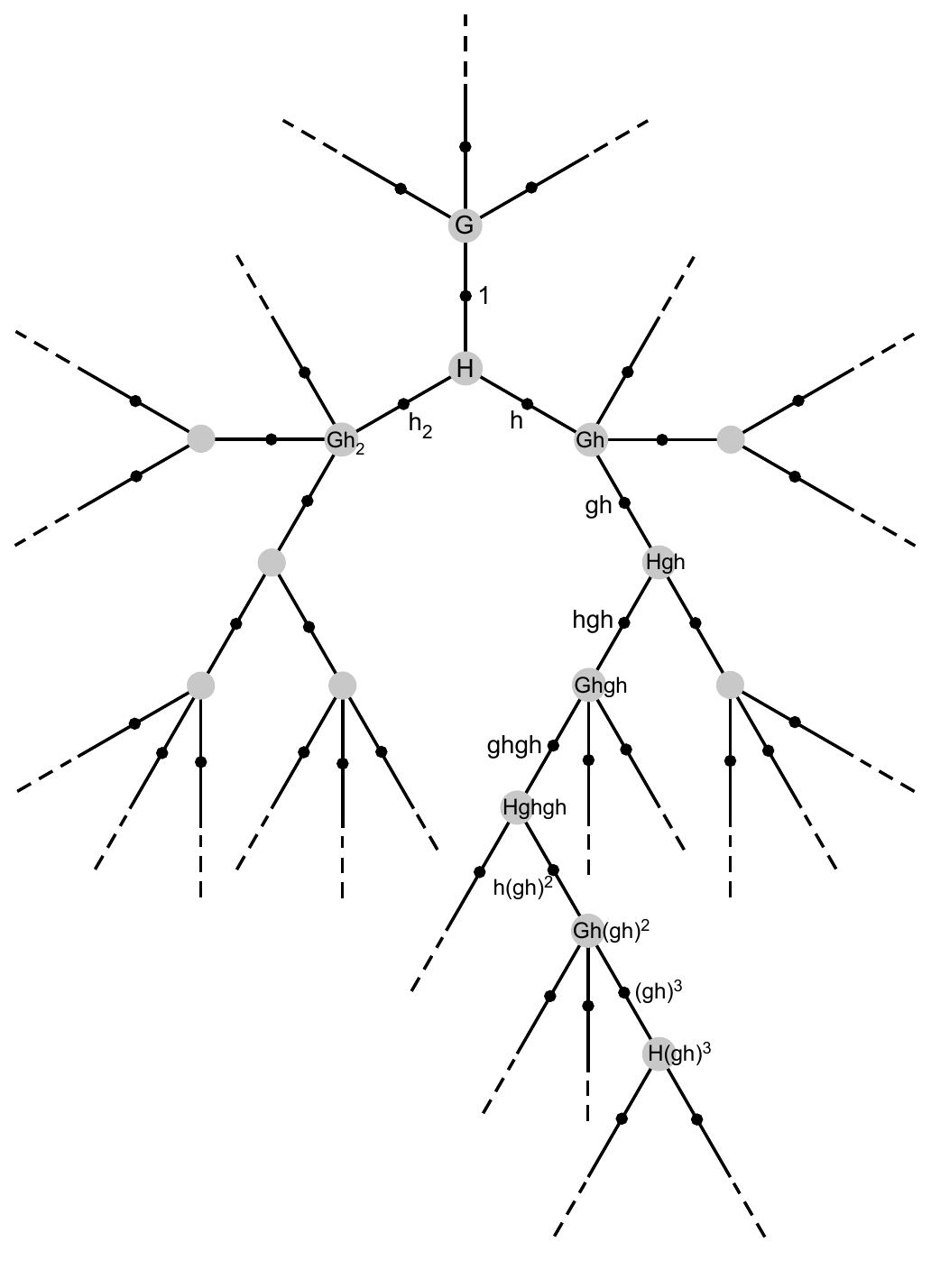}
\caption{The image of $G*H$ in the Bass-Serre tree}
\label{BStree}
\end{figure}

\begin{definition}
Let $Z$ be a real tree and $p$, $q\in Z$. We call \emph{shadow of $q$ at $p$} the set of the points $z\in Z$ such that the geodesic from $p$ to $z$ passes the point $q$ (cf. Figure  \ref{ombre}). We will denote it by $\shadow(q)_p$.
\end{definition}

\begin{figure}
\centering
\includegraphics[width=7.5cm]{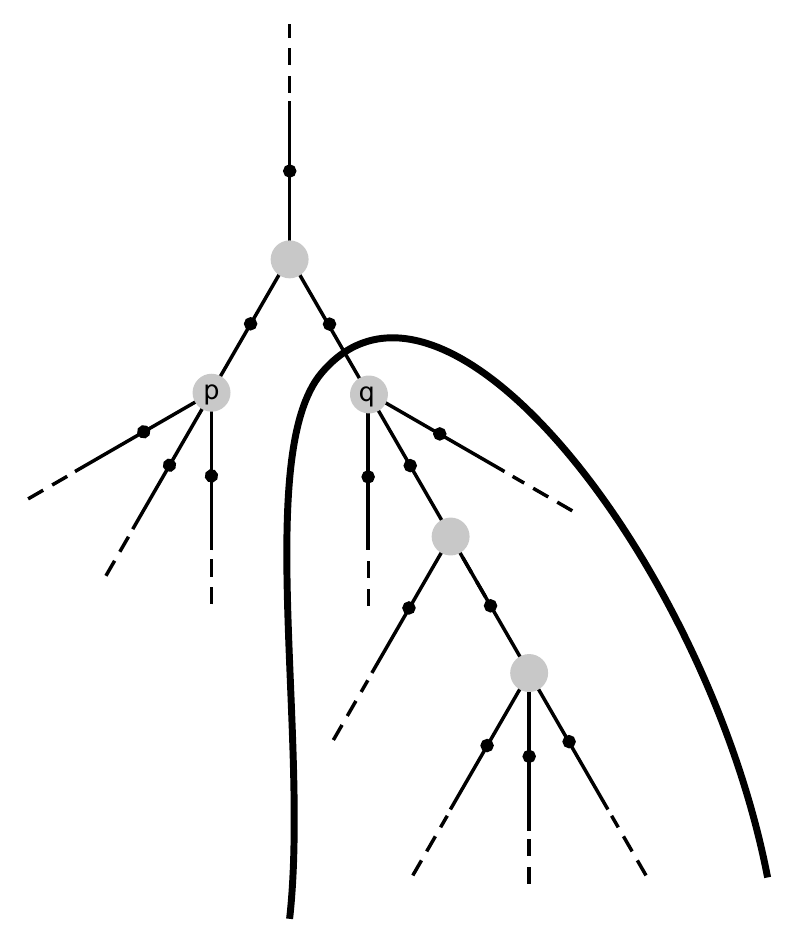}
\caption{$\shadow(q)_p$}
\label{ombre}
\end{figure}

Remark that $\shadow(q)_p$ is a subtree of $Z$ and that $q$ is the closest point to $p$ in this subtree. In addition, it is easy to see the following properties:
\begin{enumerate}
\item[$\bullet$] if $r$ is in $\shadow(q)_p$, then $\shadow(r)_p$ is contained in $\shadow(q)_p$.
\item[$\bullet$] two shadows $\shadow(q)_p$ and $\shadow(q')_p$ are either disjoint or nested.
\end{enumerate}

\bigskip

Let $T_{+}:= \shadow(H)_1$ be the shadow (of the image) of the vertex $H$ at $1$ in $T$ and let
$$
Y=T_+\cap (G\ast H).
$$
Then $Y= \bigsqcup_w Gw$ where $w$ runs in the set of non trivial elements of $G*H$ such that the normal form starts and terminates with an element of $H$. Let
$$
\Ybar = Y\cup\{1\}.
$$
Then $ \Ybar = \bigsqcup_w Hw$ where $w$ runs in the set of elements of $G*H$ such that the normal form of $w$ is either $1$, or starts with an element of $G$ and terminates with an element of $H$. Therefore, $Y$ is invariant under $G$-action (by left multiplication) and $\Ybar$ is invariant under $H$-action.

\subsection{Main result of this Section}

Let
\begin{enumerate}
\item[$\bullet$] $\Ac=\{\sigma:H\rightarrow \sym(X) \,|\, \langle G,\sigma(H)\rangle \curvearrowright X \textrm{ is admissible }\}$
\item[$\bullet$]$\Hc=\{\sigma:H\rightarrow \sym(X) \,|\, \langle G,\sigma(H)\rangle \curvearrowright X \textrm{ is highly transitive }\}$
\item[$\bullet$] $\Fc=\{\sigma:H\rightarrow \sym(X) \,|\, \langle G,\sigma(H)\rangle \curvearrowright X \textrm{ is faithful }\}$
\end{enumerate}
and recall that $\Ac$ is a Baire space (see Section \ref{PrelimBaire}).
\begin{theorem}\label{ThmFiniteFactor}
If $H$ is finite and $|H|\geq 3$, then $\Ac\cap \Hc \cap \Fc$ is generic in $\Ac$.
\end{theorem}

For $w\in G*H$, $k\in\N^*$ and $\bar x,\bar y\in X^{(k)}$, where $\bar x = (x_1,\ldots,x_k)$ and $\bar y = (y_1,\ldots, y_k)$, we put
 \begin{eqnarray*}
   \Uc_{k,\bar x,\bar y} &=& \{\sigma\in\Ac: \exists \tau\in\langle G,\sigma(H) \rangle \text{ such that } \tau(x_j) = y_j \ \forall j=1,\ldots,k\} \\
   \Uc'_w &=& \{\sigma\in\Ac: w^\sigma \neq 1 \text{ in } \sym(X)\}
 \end{eqnarray*}
Then we have
$$
  \Ac\cap\Hc = \bigcap_{k\in\N^*} \bigcap_{\bar x,\bar y\in X^{(k)}} \Uc_{k,\bar x,\bar y}
$$
and
$$
\Ac\cap\Hc\cap\Fc = \left(\bigcap_{k\in\N^*} \bigcap_{\bar x,\bar y\in X^{(k)}} \Uc_{k,\bar x,\bar y}\right) \cap \left(\bigcap_{w\in (G*H)\setminus\{1\}} \Uc'_{w}\right) \, .
$$
So it is enough to prove that the sets $\Uc_{k,\bar x,\bar y}$ and $\Uc'_w$ are open and dense in $\Ac$.

Since $\Uc_{k,\bar x,\bar y}=\cup_{w\in G*H} \mathcal{O}_{k,\bar x,\bar y,w}$ where $\mathcal{O}_{k,\bar x,\bar y,w} =  \{\sigma\in\Uc: w^\sigma(x_j) = y_j \ \forall j=1,\ldots,k\}$ which is open, the set $\Uc_{k,\bar x,\bar y}$ is open. Furthermore the set $\Uc'_{w}$ is clearly open.

We shall now prove that $\Uc_{k,\bar x,\bar y}$ and $\Uc'_{w}$ are dense in $\Ac$. We fixe from now on $k\in\N^*$, $\bar x,\bar y\in X^{(k)}$ and $F$ a finite subset of $X$. Let $\sigma\in\Ac$. To see that the set $\Uc_{k,\bar x,\bar y}$ is dense, we need to show that there exists $\alpha\in \Uc_{k,\bar x,\bar y}$ such that $\alpha|_F=\sigma|_F$. By taking a bigger finite set containing $F$ if necessary, we can suppose that $x_1,\ldots,x_k,y_1,\ldots,y_k$ are contained in $F$. Let
$$
K=\bigcup_{z\in F}\sigma(H)\cdot z.
$$
Since $F$ and $H$ are finite, $K$ is also finite. Additionally let
$$
\Kbar=\bigcup_{z\in K}G\cdot z.
$$
Notice that $\Kbar$ is infinite if $G$ is infinite, but it has finitely many $G$-orbits. Remark that $\Kbar\setminus K$ is not empty since otherwise $K$ would be formed with finite $\langle G, \sigma(H) \rangle$-orbits which contradicts to the assumption that $\sigma$ is in $\Ac$.

Recall that $T_+$ is the shadow of $H$ at $1$ in $T$ and $Y=T_+\cap(G*H)$. Since $X\setminus \Kbar$ is formed by infinitely many $G$-orbits (i.e. infinitely many copies $G_j$), there exists a $G$-equivariant bijection between $Y\times(\Kbar\setminus K)$, where $G$ acts trivially on the second factor, and $X\setminus \Kbar$. We can then extend this to a bijection $\phi$ between $\Ybar\times (\Kbar\setminus K)$ and $X\setminus K$ by sending $(1,z)$ on $z$ for every $z\in\Kbar\setminus K$. Henceforth, we denote by $Y_z$ (resp. $\Ybar_z$), the image of $Y\times \{z\}$ (resp. $\Ybar\times \{z\}$) in $X\setminus K$.

Since $K$ is $\sigma(H)$-invariant, we can define an action $\beta:H\to \sym(X)$ as follows: (cf. Figure \ref{ActionBeta}):
\begin{enumerate}
 \item[$\bullet$] $\beta|_K = \sigma|_K$;
 \item[$\bullet$] for every $z\in\Kbar\setminus K$, the restriction of $\beta$ to $\Ybar_z$ corresponds to the action of $H$ on $\Ybar\times \{z\}$ by left multiplication on the first factor.
\end{enumerate}
\begin{figure}
\centering
\includegraphics[width=11cm]{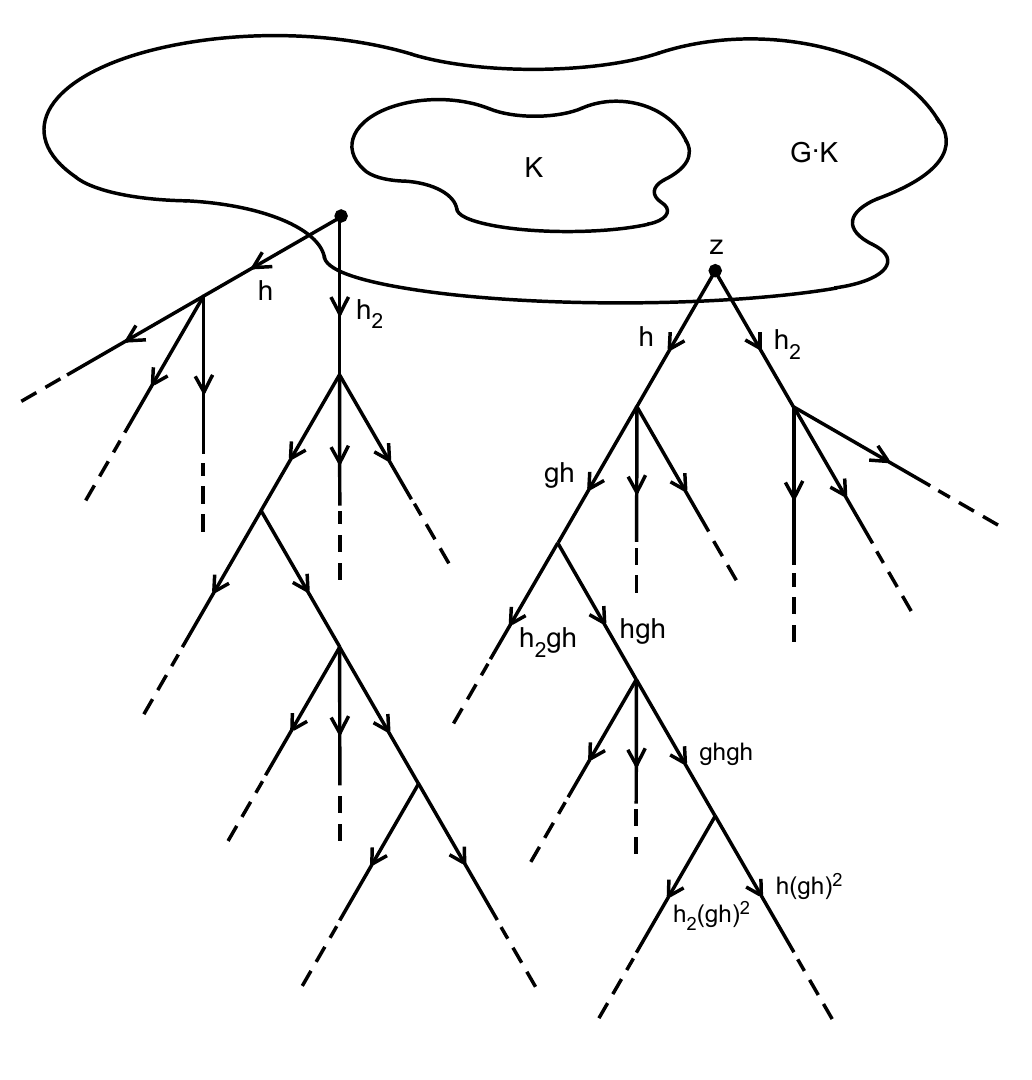}
\caption{Schreier graph of the action associated to $\beta$}
\label{ActionBeta}
\end{figure}

\begin{Claim}\label{Claim1}
The action $\beta$ is in $\Ac$.
\end{Claim}
\begin{proof}
The $\langle G, \beta(H) \rangle$-orbits are infinite since for the points in $\Ybar_z$, it follows from the construction, and for the points in $K$, it is because the $\langle G, \sigma(H) \rangle$-orbits are infinite and thus $\beta\in \Ac$.. 
\end{proof}
The action of $\beta$ induces an action of $G*H$ on $X$. Denote
by $X_\beta$ the Schreier graph of this action with respect to the generating set $G\cup H$ and by $d_{\beta}$ the distance on $X_{\beta}$. Remark that
\begin{enumerate}
 \item[$\bullet$] for every $z\neq z'$ in $\Kbar\setminus K$, there is no edge of $X_\beta$ that links an element of $Y_z$ and an element of $Y_{z'}$;
 \item[$\bullet$] the edges of $X_\beta$ that link $\Kbar$ to a subset $Y_z$ are labeled by elements of $H$, and they link  $z=\phi(1,z)$ to an element of the form $\phi(h,z)$ with $h\in H\setminus\{1\}$;
  \item[$\bullet$] the restriction of the distance $d_\beta$ to $\Ybar_z$ corresponds via $\phi^{-1}$ to the right invariant word metric on $\Ybar$.
\end{enumerate}

Since $\Ybar$ embeds isometrically in the real tree $T$, each $\Ybar_z$ can be embedded isometrically into a real tree $T_z$, and we can moreover require that no subtree of $T_z$  contains the image of $\Ybar_z$. This real tree $T_z$ is essentially unique (cf. for example Lemma~2.13 in \cite{Bestvina}). Notice that $G$ and $H$ do not act on the union of $X$ and the trees $T_z$.

\begin{Claim}\label{Claim2}
Let $w\in G*H$ and $x\in K$. Suppose that the $\beta$-trajectory of $w$ from $x$ is not contained in $K$ and let $z=s_j(w)^\beta(x)$ be the first point of this trajectory that is outside of $K$. Then $z$ is contained in $\overline{K}\setminus K$ and the end of this trajectory is a geodesic sequence in $\Ybar_z$.
Therefore, we have
$$
d(z,s_n(w)^{\beta}(x))< d(z,s_m(w)^{\beta}(x)),
$$
for every $j\leq n< m\leq |w|$.
\end{Claim}
\begin{proof}
Let us write $w=a_{|w|}\cdots a_1$ as the normal form. By hypothesis, we have
$$
y:=(a_{j-1}\cdots a_1)^\beta(x)\in K \text{ and } z=a_j^\beta(y)\notin K.
$$
Since $K$ is $\beta(H)$-invariant, $a_j$ is in $G$, $a_{j+1}$ is in $H$ and $a_{j+2},\ldots,a_{|w|}$ are alternatively in $G$ and $H$. The end of the $\beta$-trajectory of the word $a_{|w|}\cdots a_{j+1}$ from $z$ satisfies
\[(a_{\ell}\cdots a_{j+1})^\beta(z)=\phi(a_{\ell}\cdots a_{j+1},z)
\]
 for every  $\ell=j+1,\ldots,|w|$. Thus this trajectory is a geodesic sequence in $\Ybar_z$ and this proves the claim.
\end{proof}

\begin{Claim}\label{Claim3}
There exist $v_1$, $v_2\in G\ast H$ such that
\begin{enumerate}
\item their normal forms start with an element of $G$;
\item the sets $K$, $v_1^\beta(K)$ and $v_2^\beta(K)$ are pairwise disjoint.
\end{enumerate}
\end{Claim}

\begin{proof}
Since $\beta$ is in $\Ac$, by Lemma \ref{DeplacementPartiesFinies}, there exists $u_1\in G*H$ such that $u_1^\beta(K)\cap K=\emptyset$. Let $g\in G\setminus\{1\}$ and $h\in H\setminus\{1\}$. If the normal form of $u_1$ starts with an element of $G$, we put $v_1 := u_1$; otherwise, we put $v_1 := gu_1$. In both cases, the normal form of $v_1$ starts with an element of $G$. In addition, for every $x\in K$, the $\beta$-trajectory of $v_1$ from $x$ passes the point $u_1^\beta(x)$, which is not in $K$. Thus by Claim \ref{Claim2}, we have $v_1^\beta(K)\cap K= \emptyset$. Let
$$
d:=\diam(K \cup v_1^\beta(K)) \text{ and } v_2 := (gh)^{2d}v_1.
$$
The normal form of $v_2$ starts with an element of $G$. Furthermore, for every $x\in K$, the $\beta$-trajectory of $v_2$ from $x$ passes the point $v_1^\beta(x)$, which is not in $K$. So by Claim \ref{Claim2}, we have
$$
d(v_2(K), K)\geq 2d,
$$
thus the sets $K$, $v_1^\beta(K)$ and $v_2(K)$ are pairwise disjoint. This concludes the claim.
\end{proof}

Given a point $x\in X\setminus K$, there exists a unique point $z=z_x\in \Kbar\setminus K$ such that $x$ is in $\Ybar_{z}$. For the rest of the proof, we denote by $\shadow(x):=\shadow(x)_{z}$ the shadow of $x$ at $z$ in $T_z$.

\begin{Claim}\label{Claim4}
Let $M$ be a finite subset of $X\setminus K$ such that every element $y\in M$ can be written as $y=v_y^\beta(x_y)$, where $x_y\in K$ and the normal form of $v_y\in G\ast H$ starts with an element of $G\setminus\{1\}$. Then there exists $w\in G \ast H$ such that
\begin{enumerate}
\item[$\bullet$] the normal form of $w$ starts with an element of $G$ and terminates with an element of $H$;
\item[$\bullet$] $w^\beta(M)\cap K=\emptyset$;
\item[$\bullet$] $\shadow(p)\cap \shadow(p')=\emptyset$, for every $p\neq p'$ in $w^\beta(M)$.
\end{enumerate}
\end{Claim}
\begin{proof}
Let $g\in G\setminus\{1\}$ and $h\neq h'\in H\setminus\{1\}$ (recall that $H$ has at least 3 elements). Let $y\neq y'\in M$. Then,
$$
\shadow((gh)^\beta(y))\subseteq \shadow(y) \text{ and } \shadow((gh)^\beta(y'))\subseteq \shadow(y').
$$
Thus if $\shadow(y)\cap \shadow(y')=\emptyset$, then $\shadow((gh)^\beta(y))\cap\shadow((gh)^\beta(y'))$ is also empty. Notice that this argument works as well if we take $w=gh'$.

Now let us suppose that $\shadow(y)\cap \shadow(y')\neq\emptyset$. Without loss of generality, we suppose that $\shadow(y')$ is contained in $\shadow(y)$. Notice that $d(y,y')\geq 2$ since $y$, $y'\in M$ and $y\neq y'$. Let $h_1\in H$ and $g_1\in G$ the labels of the first two edges of the geodesic from $y$ to $y'$ in $X_{\beta}$. One of the elements $h$, $h'$, say $h$, is different from $h_1$. Thus $\shadow((gh)^\beta(y))$ is disjoint to $\shadow(y')$ and $\shadow((gh)^\beta(y'))$.

Therefore, if given a finite subset $S\subset X\setminus K$ we denote by $n_s(S)$ the number of pairs $(q,q')\in S\times S$ such that $\shadow(q)\cap \shadow(q')\neq \emptyset$, then we have
$$
n_s((gh)^\beta(M))< n_s(M) \text{ or } n_s((gh')^\beta(M))< n_s(M).
$$
In addition, Claim \ref{Claim2} guaranties that $(gh)^\beta(M)$ and $(gh')^\beta(M)$ do not intersect with $K$. By repeating this operation at most $|M|$ times, we obtain an element $w$ as we wished.
\end{proof}

\begin{proof}[End of the proof of Theorem \ref{ThmFiniteFactor}.]
Now let
$$
M:=v_1^\beta(K)\sqcup v_2^\beta(K)
$$
where $v_1$, $v_2$ are the elements as in Claim \ref{Claim3}. Then there is $w$ as in Claim \ref{Claim4}. We thus have two elements $w_j = wv_j \in G*H$ ($j=1,2$) such that
\begin{enumerate}
\item[$\bullet$] the normal form of $w_j$ ($j=1,2$) starts with an element of $G$;
\item[$\bullet$] the set $\Kbar$ and the shadows of the elements of $w_1^\beta(K) \sqcup w_2^\beta(K)$ are pairwise disjoint.
\end{enumerate}
In addition, the $\beta$-trajectory of $w_1$ and $w_2$ from the points in $K$ do not intersect with the shadows of the points of $w_1^\beta(K) \sqcup w_2^\beta(K)$ before their last points, since as soon as the $\beta$-trajectories leave $K$, they are geodesic lines by Claim \ref{Claim2}.

We then define an action $\alpha$ by modifying $\beta$ as follows (cf. Figure \ref{ActionAlpha}). We choose $k(|H|-2)$ points $p_{i,j}$ (where $3\leq i \leq |H|$ and $1\leq j \leq k$), outside of $K$ and such that:
\begin{enumerate}
 \item[$\bullet$] their $G$-orbits are pairwise disjoint;
 \item[$\bullet$] for each $i$, $j$, the $H$-orbit of $p_{i,j}$ is contained in $\shadow(p_{i,j})$;
 \item[$\bullet$] $\shadow(p_{i,j})$ does not intersect with the $\beta$-trajectory of $w_1$ and of $w_2$ from the points of $K$.
\end{enumerate}

 Let $h$ be a non trivial element of $H$. For each $j=1,\ldots,k$, we put $H$ in bijection with $$
 A_j = \{w_1^\beta(x_j),w_2^\beta(y_j), p_{3,j},\ldots,p_{|H|,j}\}
 $$
in such a way that $w_1^\beta(x_j)$ corresponds to $1$ and $w_2^\beta(y_j)$ corresponds to $h$. We then transfer the $H$-action by left multiplication in order to define $\alpha$ on the $A_j$'s. The other points of $\beta(H)$-orbits of the families $A_j$, are defined as fixed points under the $\alpha(H)$-action. For every other points, we set $\alpha(x) = \beta(x)$ so that we have in particular $\alpha|_F=\beta|_F$.

\begin{figure}
\centering
\includegraphics[width=13cm]{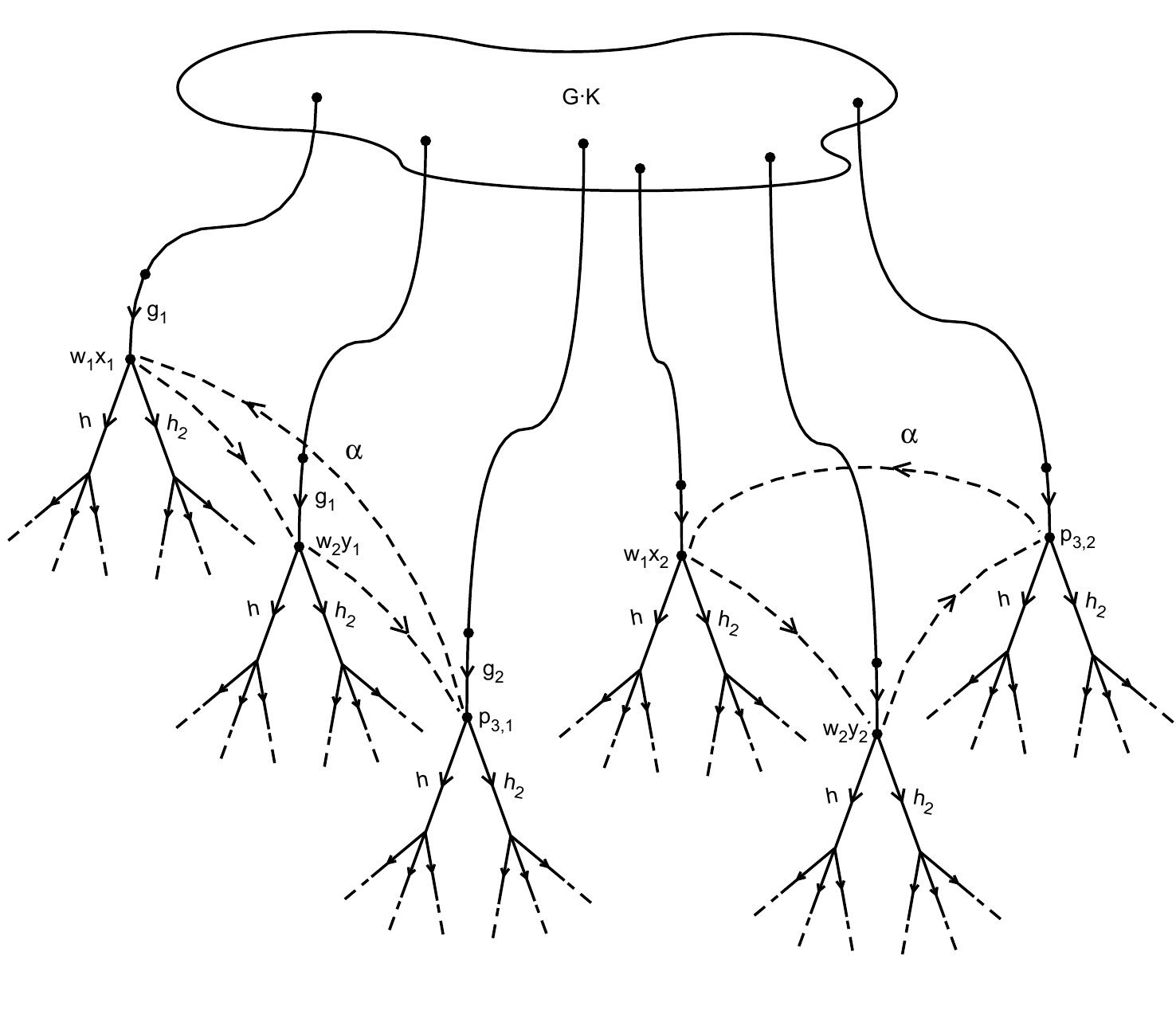}
\caption{Schreier graph of the action associated to $\alpha$}
\label{ActionAlpha}
\end{figure}

By construction, the $\langle G,\alpha(H)\rangle$-orbits are infinite since the Schreier graph of each orbit contains an infinite subtree. Furthermore we have
$$
(w_2^{-1} h w_1)^\alpha(x_j) = y_j
$$
for every $j$, and therefore $\alpha\in\Uc_{k,\bar x, \bar y}$. Besides, $\sigma$, $\beta$ and $\alpha$ coincide on $F$, which proves that $\Uc_{k,\bar x, \bar y}$ is dense in $\Ac$.

Finally, if $\sigma$ is in $\Ac$ and $F$ is the finite set of $X$ as before, then this $\beta$-action is faithful as well and $\sigma|_F=\beta|_F$. Therefore $\Uc'_w$, for $w\in (G*H)\setminus\{1\}$, are dense. This achieves the proof of Theorem \ref{ThmFiniteFactor}.
\end{proof}

\begin{proof}[Proof of Theorem \ref{main}.]
In case $G\simeq\Z/2\Z\simeq H$, the group $G*H$ is isomorphic to the infinite dihedral group, which has trivial center but it contains a cyclic subgroup of index 2. Hence $G*H$ does not admit any faithful and highly transitive action by Corollary \ref{GeneralFactsFHT}.

If at least one of the factors $G$, $H$ is not isomorphic to the cyclic group $\Z/2\Z$, By Theorem \ref{FreeProductInfinite} and Theorem \ref{ThmFiniteFactor} it admits a faithful and highly transitive actions.
\end{proof}


\end{document}